\documentclass[a4paper,11pt]{amsart}

\usepackage[a4paper,left=28mm,right=28mm,top=28mm,bottom=28mm]{geometry}
\usepackage{mathtools}
\usepackage{amssymb}
\usepackage{tikz}
\usetikzlibrary{arrows.meta}
\usepackage{aliascnt}
\usepackage[colorlinks=true,linkcolor=blue,citecolor=blue]{hyperref}

\newtheorem{maintheorem}{Theorem}

\newaliascnt{lemma}{theorem}
\newtheorem{lemma}[lemma]{Lemma}
\aliascntresetthe{lemma}

\newaliascnt{proposition}{theorem}
\newtheorem{proposition}[proposition]{Proposition}
\aliascntresetthe{proposition}

\newaliascnt{corollary}{theorem}

\aliascntresetthe{corollary}

\theoremstyle{definition}

\newaliascnt{definition}{theorem}
\newtheorem{definition}[definition]{Definition}
\aliascntresetthe{definition}

\newaliascnt{example}{theorem}

\aliascntresetthe{example}

\theoremstyle{remark}

\newaliascnt{remark}{theorem}
\newtheorem{remark}[remark]{Remark}
\aliascntresetthe{remark}

\newcommand{\ts}{\mathcal{X}_{\Sigma,\beta}}

\title{The Integral Chow Ring of Smooth Non-strict Toric Stacks}

\author{Sota Suzuki}

\address{Sota Suzuki, Department of Mathematics, Graduate School of Science, The University of Osaka}
\email{u956309e@ecs.osaka-u.ac.jp}

\begin{document}

\begin{abstract}
  We extend the Stanley--Reisner ring to the non-strict stacky fans introduced by Geraschenko and Satriano.
  We then prove that this ring is isomorphic to the integral Chow ring of the associated toric stack when the underlying fan is smooth.
  More generally, we show that this ring is isomorphic to the integral Chow ring of the canonical stack over the associated toric stack for an arbitrary non-strict stacky fan.
  These generalize earlier results on the integral Chow rings of smooth toric varieties and several classes of toric stacks.
  We also give an explicit description of the pullback homomorphism induced by a morphism of non-strict stacky fans.
\end{abstract}

\maketitle

\setcounter{tocdepth}{1}

\section{Introduction}

The Chow ring of a smooth toric variety has a well-known Stanley--Reisner presentation (see \cite{Gross}), and this presentation has been generalized to toric stacks.
There are, however, several different notions of toric stacks in the literature.
For example, toric stacks in the sense of \cite{Borisov_2004} are smooth Deligne--Mumford stacks, and based on this construction, a class of smooth toric Artin stacks with trivial generic stabilizer was defined in \cite{integration}.
Geraschenko and Satriano gave a general definition of toric stacks as quotient stacks; this definition includes these and other classes \cite{GS1}.  
In this paper, we work with non-strict toric stacks as defined there.

The integral Chow ring of a smooth quotient stack $[X/G]$ is defined as the equivariant Chow ring $A_G^*(X)$ \cite{edidin-graham}.
For a smooth toric  Deligne--Mumford stack in the sense of \cite{Borisov_2004} with a semi-projective coarse moduli space, the integral Chow ring was computed in \cite{jiang2007integralorbifoldchowring}, using a result of Iwanari \cite{iwanari2007integralchowringstoric}.
In \cite{integration}, the integral Chow ring was computed for the smooth toric Artin stacks defined there.
Also, the integral Chow ring of the canonical stack of a toric variety associated with a full-dimensional fan was computed in \cite{ER}.
The purpose of this paper is to compute the integral Chow ring of a smooth non-strict toric stack in the sense of \cite{GS1}.
Our main theorem generalizes these computations.

\subsection{Main result}

Throughout this paper, we work over an algebraically closed field $k$ of characteristic $0$.
We write $D(-)$ for Cartier duality between finitely generated abelian groups and diagonalizable group schemes over $k$.
For a diagonalizable group scheme $G$, $D(G)$ is its character group.

A \textit{non-strict stacky fan} is a pair $(\Sigma,\beta\colon L\to N)$, where $\Sigma$ is a fan in $L_\mathbb{R}$ for a lattice $L$, and $\beta\colon L\to N$ is a homomorphism to a finitely generated abelian group $N$.
When $N$ is a lattice and $\mathrm{coker}(\beta)$ is finite, it is called a \textit{strict stacky fan}.
Let $X_\Sigma$ be the toric variety defined by $\Sigma$, with dense torus $T_L=D(L^*)$. 
The homomorphism $\beta$ determines a diagonalizable group scheme $G_\beta$ and a homomorphism $\delta_\beta\colon L^*\to D(G_\beta)$.
Taking the Cartier dual of $\delta_\beta$ gives a homomorphism $G_\beta\to T_L$, through which $G_\beta$ acts on $X_\Sigma$.
This action defines the \textit{non-strict toric stack} $\ts=[X_\Sigma/G_\beta]$.

For each $\rho\in \Sigma(1)$, let $u_\rho\in L$ denote the primitive generator of $\rho$, and let $D_\rho$ denote the corresponding torus-invariant prime divisor on $X_\Sigma$.
Since $G_\beta$ acts on $X_\Sigma$ through the homomorphism $G_\beta\to T_L$, the divisor $D_\rho$ is $G_\beta$-invariant.
Set $R_\beta=\mathrm{Sym}_\mathbb{Z}D(G_\beta)$, and regard $\delta_\beta$ as a homomorphism $L^*\to R_\beta$.
When $\Sigma$ is smooth, the pullback along the structure morphism $X_\Sigma\to \mathrm{Spec}(k)$ makes $A^*(\mathcal{X}_{\Sigma,\beta})$ a graded $R_\beta$-algebra via the isomorphism $R_\beta\cong A^*_{G_\beta}(\mathrm{Spec}(k))$ of \eqref{equbg}.

\begin{maintheorem}\label{thmmain}
  Let $(\Sigma,\beta)$ be a non-strict stacky fan with $\Sigma$ smooth.
  Then there is an isomorphism of graded $R_\beta$-algebras 
  \[
  A^*(\ts)\cong \frac{R_\beta[x_\rho\mid \rho\in \Sigma(1)]}{I_\Sigma+J_{\Sigma,\beta}},
  \]
  where
  \begin{align*}
    I_\Sigma&=(x_{\rho_1}\cdots x_{\rho_s}\,\mid\, \rho_1,\dots,\rho_s\ \text{do\ not\ span\ a\ cone\ of}\ \Sigma),\\
    J_{\Sigma,\beta}&=\left(\sum_{\rho\in\Sigma(1)}\langle m,u_\rho\rangle x_\rho-\delta_\beta(m)\;\middle|\; m\in L^*\right).
  \end{align*}
  Under this isomorphism, $x_\rho$ corresponds to $c_1^{G_\beta}(\mathcal{O}(D_\rho))=[D_\rho]_{G_\beta}$.
\end{maintheorem}

More generally, we show that for any non-strict stacky fan $(\Sigma,\beta)$, the integral Chow ring of the canonical stack over $\mathcal{X}_{\Sigma,\beta}$ in the sense of \cite[Definition~5.1]{GS1} is isomorphic as a graded $R_\beta$-algebra to 
\[
\mathrm{SR}(\Sigma,\beta)\coloneqq\frac{R_\beta[x_\rho\mid\rho\in\Sigma(1)]}{\mathcal{I}_\Sigma+J_{\Sigma,\beta}},
\]
where $\mathcal{I}_\Sigma$ is generated by the monomials $\prod_{\rho\in C}x_\rho$ for all primitive collections $C$ of $\Sigma$, i.e., $C$ is not contained in a cone of $\Sigma$ but every proper subset of $C$ is.
This isomorphism with $\mathrm{SR}(\Sigma,\beta)$ generalizes the corresponding result for canonical stacks of toric varieties \cite[Theorem~3.6]{ER}.
In \autoref{subsecpull}, we also determine the pullback homomorphism 
\[
\mathrm{SR}(\Sigma',\beta')\longrightarrow \mathrm{SR}(\Sigma,\beta)
\]
induced by a morphism $(\Sigma,\beta)\to (\Sigma',\beta')$ of non-strict stacky fans with $\Sigma$ and $\Sigma'$ smooth.

After the first version of this paper was posted on arXiv, Singh gave an independent proof of the corresponding result for smooth strict toric stacks in \cite{Singh2026}.

\subsection*{Acknowledgements}

I would like to thank my supervisor, Takehiko Yasuda, for valuable comments and
suggestions on this paper.

\section{Preliminaries}

\subsection{Non-strict toric stacks}\label{subsects}

In this subsection, we recall the construction of non-strict toric stacks from \cite{GS1}.
Let $(\Sigma,\beta)$ be a non-strict stacky fan.
Let $C(\beta)$ be the mapping cone of $\beta$, and let $C(\beta)^*=\mathbf{R}\mathrm{Hom}_\mathbb{Z}(C(\beta),\mathbb{Z})$ be its derived dual.
This yields the exact sequence
\[
 0\longrightarrow H^0(C(\beta)^*)\longrightarrow N^*\longrightarrow L^*\longrightarrow H^1(C(\beta)^*)\longrightarrow \mathrm{Ext}_\mathbb{Z}^1(N,\mathbb{Z})\longrightarrow 0.
\]
For $i=0,1$, we define $G_\beta^i$ to be the diagonalizable group scheme corresponding to $D(G_\beta^i)\coloneqq H^i(C(\beta)^*)$, and we define $G_\beta=G_\beta^0\oplus G_\beta^1$.
The composition of the natural morphism $L^*\to D(G_\beta^1)$ and the inclusion $D(G_\beta^1)\hookrightarrow D(G_\beta)$ defines $L^*\to D(G_\beta)$.
By Cartier duality, this homomorphism corresponds to a homomorphism $G_\beta\to D(L^*)=T_L$, which defines an action of $G_\beta$ on $X_\Sigma$ through the torus action.
The associated non-strict toric stack is $\ts=[X_\Sigma/G_\beta]$.

We can also explicitly determine $G_\beta^0$ and $G_\beta^1$.
By fixing isomorphisms $L\cong \mathbb{Z}^r$ and $N\cong \mathbb{Z}^d\oplus\mathbb{Z}/a_1\mathbb{Z}\oplus\cdots\oplus\mathbb{Z}/a_s\mathbb{Z}$, we obtain free resolutions
\begin{align*}
  \dot{L} : 0\longrightarrow \mathbb{Z}^r\longrightarrow 0 \quad \text{and} \quad
  \dot{N} : 0\longrightarrow \mathbb{Z}^s\stackrel{Q}{\longrightarrow}\mathbb{Z}^{d+s}\longrightarrow 0,
\end{align*}
where 
\[
Q=
\begin{pmatrix}
  0_{d\times s} \\ \mathrm{diag}(a_1,\dots,a_s)
\end{pmatrix}.
\]
Choose a lift $B\colon \mathbb{Z}^r\to \mathbb{Z}^{d+s}$ of $\beta$.
It follows from \cite[Remark~2.18]{GS1} that $G_\beta^0\cong D(\mathrm{ker}([B,Q]^*))$ and $G_\beta^1\cong D(\mathrm{coker}([B,Q]^*))$.
Under these isomorphisms, the map $L^*\to D(G_\beta^1)$ is given by the composite of the inclusion $(\mathbb{Z}^r)^*\hookrightarrow (\mathbb{Z}^{r+s})^*$ and the quotient map $(\mathbb{Z}^{r+s})^*\to \mathrm{coker}([B,Q]^*)$.

We also recall the notion of a morphism of non-strict stacky fans.

\begin{definition}[{\cite[Definition~3.2]{GS1}}]\label{defmor}
  A \textit{morphism of non-strict stacky fans}
  \[
  (\Sigma,\beta\colon L\to N)\longrightarrow (\Sigma',\beta'\colon L'\to N')
  \]
  is a pair of group morphisms $\Phi\colon L\to L'$ and $\phi\colon N\to N'$ so that $\beta'\circ\Phi=\phi\circ\beta$ and so that for every cone $\sigma\in \Sigma$, $\Phi(\sigma)$ is contained in a cone of $\Sigma'$.
\end{definition}

A morphism of non-strict stacky fans $(\Phi,\phi)$ induces a morphism of toric varieties $f_\Phi\colon X_\Sigma\to X_{\Sigma'}$ and a compatible homomorphism $g\colon G_\beta\to G_{\beta'}$, and hence a morphism
\[
\ts\longrightarrow \mathcal{X}_{\Sigma',\beta'}.
\]

Suppose that $\mathrm{coker}(\beta)$ and $\mathrm{coker}(\beta')$ are finite. We describe the induced homomorphism $D(g)\colon D(G_{\beta'})\to D(G_\beta)$ more concretely. 
Write $N'\cong\mathbb Z^e\oplus\mathbb Z/b_1\mathbb Z \oplus\cdots\oplus\mathbb Z/b_t\mathbb Z$, and let $B'$ and $Q'$ be defined as above. 
Choose a lift $\Psi\colon\mathbb Z^{d+s}\to\mathbb Z^{e+t}$ of $\phi$.
Then there is a lift $F\colon L\oplus\mathbb Z^s\to L'\oplus\mathbb Z^t$ of $\Phi$ with respect to the projections onto $L$ and $L'$ such that 
\[
\Psi[B,Q]=[B',Q']F.
\]
Hence we obtain the following commutative diagram with exact rows:
\begin{equation}\label{equdiag}\vcenter{\hbox{
\begin{tikzpicture}[>=stealth]
  \node (A1) at (0,1.4) {$(\mathbb Z^{e+t})^*$};
  \node (A2) at (3.6,1.4) {$(L'\oplus\mathbb Z^t)^*$};
  \node (A3) at (6.7,1.4) {$D(G_{\beta'})$};
  \node (A4) at (8.4,1.4) {$0$};
  \node (B1) at (0,0) {$(\mathbb Z^{d+s})^*$};
  \node (B2) at (3.6,0) {$(L\oplus\mathbb Z^s)^*$};
  \node (B3) at (6.7,0) {$D(G_\beta)$};
  \node (B4) at (8.4,0) {$0$\rlap{.}};

  \draw[->] (A1) -- node[above] {$\scriptstyle [B',Q']^*$} (A2);
  \draw[->] (A2) -- (A3);
  \draw[->] (A3) -- (A4);
  \draw[->] (B1) -- node[below] {$\scriptstyle [B,Q]^*$} (B2);
  \draw[->] (B2) -- (B3);
  \draw[->] (B3) -- (B4);

  \draw[->] (A1) -- node[left] {$\scriptstyle\Psi^*$} (B1);
  \draw[->] (A2) -- node[left] {$\scriptstyle F^*$} (B2);
  \draw[->] (A3) -- node[left] {$\scriptstyle D(g)$} (B3);
\end{tikzpicture}}}
\end{equation}

\subsection{Canonical stacks}\label{subseccan}

We recall the canonical stack construction of \cite[Section~5]{GS1}, which is based on the Cox construction \cite[\S~5.1]{cox2024toric}.
Let $(\Sigma,\beta\colon L\to N)$ be a non-strict stacky fan. 
Let $L_\Sigma$ be the saturated sublattice of $L$ spanned by $\Sigma$, and let $M\subseteq L$ be a direct complement to $L_\Sigma$, i.e., $L=L_\Sigma\oplus M$.
We denote the inclusion by $\iota\colon M\hookrightarrow L$.
For each ray $\rho\in\Sigma(1)$, let $e_\rho$ be the corresponding standard basis element of $\mathbb{Z}^{\Sigma(1)}$.

For each $\sigma\in \Sigma$, let $\widetilde{\sigma}$ be the cone generated by $\{e_\rho\,\mid\,\rho\in\sigma(1)\}$.
We define the fan $\widetilde{\Sigma}$ on $\widetilde{L}\coloneqq \mathbb{Z}^{\Sigma(1)}\oplus M$ as the fan generated by all the $\widetilde{\sigma}$.
Define a homomorphism $\Phi\colon \widetilde{L}\to L$ by $(e_\rho,m)\mapsto u_\rho+\iota(m)$, and set $\widetilde{\beta}=\beta\circ\Phi$.
\[
\begin{tikzpicture}[>=stealth]
  \node (A) at (0,0) {$\widetilde{L}$};
  \node (B) at (2,0) {$L$};
  \node (C) at (0,-1.2) {$N$};
  \node (D) at (2,-1.2) {$N$};
  \node (E) at (0,0.5) {$\widetilde{\Sigma}$};
  \node (F) at (2,0.5) {$\Sigma$};
  \draw[->] (A) -- node[above] {$\scriptstyle \Phi$} (B);
  \draw[->] (A) -- node[left] {$\scriptstyle \widetilde{\beta}$} (C);
  \draw[->] (B) -- node[right] {$\scriptstyle \beta$} (D);
  \draw[line width=2.0pt] (C) -- (D);
  \draw[white, line width=1.0pt, line cap=rect] (C) -- (D);
  \draw[->] (E) -- (F);
\end{tikzpicture}
\]
Then the morphism of non-strict stacky fans $(\Phi,\mathrm{id}_N)\colon (\widetilde{\Sigma},\widetilde{\beta})\to (\Sigma,\beta)$ induces a toric morphism $\mathcal{X}_{\widetilde{\Sigma},\widetilde{\beta}}\to \ts$.

\begin{definition}[{\cite[Definition~5.1]{GS1}}]
  We call $\mathcal{X}_{\widetilde{\Sigma},\widetilde{\beta}}$ the \textit{canonical stack} over $\mathcal{X}_{\Sigma,\beta}$, and we say that the morphism $\mathcal{X}_{\widetilde{\Sigma},\widetilde{\beta}}\to \ts$ is a \textit{canonical stack morphism}.
\end{definition}

If $\Sigma$ is smooth, then the canonical stack morphism is an isomorphism by \cite[Remark~B.22]{GS1}.

Let $B_\Sigma=\left(\prod\nolimits_{\rho\notin \sigma(1)}X_\rho\,\middle|\,\sigma\in\Sigma \right)\subset k[X_\rho\mid\rho\in\Sigma(1)]$
be the irrelevant ideal of $\Sigma$, and set $Z(\Sigma)\coloneqq V(B_\Sigma)\subset \mathbb{A}^{\Sigma(1)}$.
By the construction of $\widetilde{\Sigma}$, we have $X_{{\widetilde{\Sigma}}}\cong (\mathbb{A}^{\Sigma(1)}\setminus Z(\Sigma))\times T_M$.
The closed subset $Z(\Sigma)$ can be described in terms of primitive collections.

\begin{lemma}[{\cite[Proposition~5.1.6]{cox2024toric}}]\label{lemint}
  We have 
  \[
  Z(\Sigma)=\bigcup_C V(X_{\rho_1},\dots,X_{\rho_s}),
  \]
  where the union is over all primitive collections $C\subset \Sigma(1)$.
\end{lemma}

The group $G_{\widetilde{\beta}}$ can also be described explicitly.
By \cite[Lemma~A.1]{GS1}, we have the following commutative diagram with exact rows:
\[
\begin{tikzpicture}[>=stealth]
  \node (A0) at (0,0) {$0$};
  \node (A1) at (1.5,0) {$G_\Phi$};
  \node (A2) at (3.5,0) {$G_{\widetilde{\beta}}$};
  \node (A3) at (5.5,0) {$G_\beta$};
  \node (A4) at (7,0) {$0$};
  \node (B0) at (0,-1.2) {$0$};
  \node (B1) at (1.5,-1.2) {$G_\Phi$};
  \node (B2) at (3.5,-1.2) {$T_{\widetilde{L}}$};
  \node (B3) at (5.5,-1.2) {$T_L$};
  \node (B4) at (7,-1.2) {$0$\rlap{.}};
  \draw[->] (A0) -- (A1);
  \draw[->] (A1) -- (A2);
  \draw[->] (A2) -- (A3);
  \draw[->] (A3) -- (A4);
  \draw[->] (B0) -- (B1);
  \draw[->] (B1) -- (B2);
  \draw[->] (B2) -- (B3);
  \draw[->] (B3) -- (B4);
  \draw[line width=2.0pt] (A1) -- (B1);
  \draw[white, line width=1.0pt, line cap=rect] (A1) -- (B1);
  \draw[->] (A2) -- (B2);
  \draw[->] (A3) -- (B3);
\end{tikzpicture}
\]
Cartier duality then yields the commutative diagram:
\[
\begin{tikzpicture}[>=stealth]
  \node (A0) at (0,0) {$0$};
  \node (A1) at (1.5,0) {$L^*$};
  \node (A2) at (3.5,0) {$\widetilde{L}^*$};
  \node (A3) at (5.5,0) {$D(G_\Phi)$};
  \node (A4) at (7,0) {$0$};
  \node (B0) at (0,-1.2) {$0$};
  \node (B1) at (1.5,-1.2) {$D(G_\beta)$};
  \node (B2) at (3.5,-1.2) {$D(G_{\widetilde{\beta}})$};
  \node (B3) at (5.5,-1.2) {$D(G_\Phi)$};
  \node (B4) at (7,-1.2) {$0$\rlap{.}};
  \draw[->] (A0) -- (A1);
  \draw[->] (A1) -- node[above] {$\scriptstyle \Phi^*$} (A2);
  \draw[->] (A2) -- (A3);
  \draw[->] (A3) -- (A4);
  \draw[->] (B0) -- (B1);
  \draw[->] (B1) -- (B2);
  \draw[->] (B2) -- (B3);
  \draw[->] (B3) -- (B4);
  \draw[->] (A1) -- node[left] {$\scriptstyle \delta_\beta$} (B1);
  \draw[->] (A2) -- (B2);
  \draw[line width=2.0pt] (A3) -- (B3);
  \draw[white, line width=1.0pt, line cap=rect] (A3) -- (B3);
\end{tikzpicture}
\]
Since the left square is a pushout, we have 
\[
D(G_{\widetilde{\beta}})\cong D(G_\beta)\oplus_{L^*}\widetilde{L}^*.
\] 
From the identification $\widetilde{L}^*=(\mathbb{Z}^{\Sigma(1)})^*\oplus M^*$ and the definition of $\Phi$, we obtain 
\[
  D(G_{\widetilde{\beta}})\cong \frac{D(G_\beta)\oplus(\mathbb{Z}^{\Sigma(1)})^*\oplus M^*}{\left(-\delta_\beta(m),\sum\nolimits_{\rho\in \Sigma(1)}\langle m,u_\rho\rangle e_\rho^*,\iota^*(m)  \, \middle| \, m\in L^*\right)}.
\]
Choose a basis $\mu_1,\dots,\mu_r$ of $M^*$, and let $v_1,\dots,v_r$ be its dual basis in $M$.
Taking symmetric algebras, we obtain 
\begin{align}\label{equdg}
  \mathrm{Sym}(D(G_{\widetilde{\beta}}))=\frac{R_\beta[x_\rho\mid \rho\in\Sigma(1)][y_1,\dots,y_r]}{\left(\sum_{\rho\in\Sigma(1)}\langle m,u_\rho\rangle x_\rho-\delta_\beta(m)+\sum_{j=1}^r\langle m,v_j\rangle y_j\,\middle|\, m\in L^*\right)},
\end{align}
where $x_\rho$ and $y_j$ correspond to $e_\rho^*$ and $\mu_j$, respectively.

\subsection{Equivariant intersection theory}

We refer to Edidin and Graham \cite{edidin-graham} for the definitions and basic properties of equivariant intersection theory.
For a diagonalizable group scheme $G$ and a character $\chi\in D(G)$, we denote by $\mathbb{A}^1_\chi$ the one-dimensional representation of $G$ defined by $g\cdot z=\chi(g)z$.
The homomorphism $D(G)\to A^1(BG)$, $\chi\mapsto c_1^G(\mathbb{A}^1_\chi)$, induces an isomorphism of graded rings
\begin{align}\label{equbg}
  \mathrm{Sym}_\mathbb{Z}(D(G))\stackrel{\sim}{\longrightarrow} A^*_G(\mathrm{Spec}(k))=A^*(BG).
\end{align}

We use the following result to compute the Chow ring of the canonical stack.
\begin{lemma}[{\cite[Proposition~2.1]{CE}}]\label{lemce}
  Let $V$ be a finite-dimensional representation of $G$, and let $L_1,\dots,L_s$ be $G$-invariant linear subspaces of $V$.
  Set $Z=V\setminus \bigcup_{i=1}^sL_i$. Then the pullback along the structure morphism $Z\to \mathrm{Spec}(k)$ induces the following isomorphism of graded rings
  \[
  A_G^*(Z)=\frac{A_G^*(\mathrm{Spec}(k))}{(c^G_{\mathrm{top}}(V/L_i)\,\mid\,1\le i\le s)},
  \]
  where $V/L_i$ denotes the quotient representation.
\end{lemma}

\section{Integral Chow ring of smooth non-strict toric stacks}

\subsection{Proof of the main theorem}

We compute the integral Chow ring of the canonical stack $\mathcal{X}_{\widetilde{\Sigma},\widetilde{\beta}}$ associated with a non-strict stacky fan $(\Sigma,\beta)$.
Since the canonical stack morphism is a morphism over $BG_\beta$, pullback makes $A^*(\mathcal{X}_{\widetilde{\Sigma},\widetilde{\beta}})$ a graded $R_\beta$-algebra.

\begin{proposition}
  Let $(\Sigma,\beta)$ be a non-strict stacky fan. Then there is an isomorphism of graded $R_\beta$-algebras
  \[
  A^*(\mathcal{X}_{\widetilde{\Sigma},\widetilde{\beta}})\cong \mathrm{SR}(\Sigma,\beta).
  \]
  Under this isomorphism, $x_\rho$ corresponds to the equivariant class $[V(X_\rho)]_{G_{\widetilde{\beta}}}$.
\end{proposition}
\begin{proof}
  Let $\mu_1,\dots,\mu_r$ be the basis of $M^*$ with dual basis $v_1,\dots,v_r$ in $M$.
  This basis identifies $T_M$ with $\mathbb{G}_m^r\subset \mathbb{A}^r=\mathrm{Spec}\,k[Y_1,\dots,Y_r]$.
  We regard $\mathbb{A}^{\Sigma(1)}\times \mathbb{A}^r$ as a $G_{\widetilde{\beta}}$-representation via $G_{\widetilde{\beta}}\to T_{\widetilde{L}}$.
  Consider the $G_{\widetilde{\beta}}$-invariant linear subspaces 
  \[
  W_C=\{X_\rho=0\,\mid\,\rho\in C\},\quad H_j=\{Y_j=0\},
  \]
  where $C$ runs over all primitive collections of $\Sigma$ and $1\le j\le r$.
  By \autoref{lemint}, their union is the complement of $(\mathbb{A}^{\Sigma(1)}\setminus Z(\Sigma))\times T_M$.
  The quotient representations are $\bigoplus_{\rho\in C}\mathbb{A}^1_{\delta_{\widetilde{\beta}}(e_\rho^*)}$ and $\mathbb{A}^1_{\delta_{\widetilde{\beta}}(\mu_j)}$, respectively.
  Therefore, \autoref{lemce}, \eqref{equdg}, and \eqref{equbg} yield 
  \begin{align*}
    &A^*_{G_{\widetilde{\beta}}}((\mathbb{A}^{\Sigma(1)}\setminus Z(\Sigma))\times T_M)\\
    \cong&\frac{A^*_{G_{\widetilde{\beta}}}(\mathrm{Spec}(k))}{\left(\prod_{\rho\in C}c_1^{G_{\widetilde{\beta}}}\big(\mathbb{A}^1_{\delta_{\widetilde{\beta}}(e_\rho^*)}\big)\,\middle|\, C\subset \Sigma(1)\right)+\left(c_1^{G_{\widetilde{\beta}}}\big(\mathbb{A}^1_{\delta_{\widetilde{\beta}}(\mu_j)}\big)\,\middle|\,1\le j\le r\right)}\\
    \cong&\frac{R_\beta[x_\rho\mid \rho\in\Sigma(1)][y_1,\dots,y_r]}{\mathcal{I}_\Sigma+\left(\sum_{\rho\in\Sigma(1)}\langle m,u_\rho\rangle x_\rho-\delta_\beta(m)+\sum_{j=1}^r\langle m,v_j\rangle y_j\,\middle|\, m\in L^*\right)+(y_1,\dots,y_r)}\\
    \cong&\frac{R_\beta[x_\rho\mid \rho\in\Sigma(1)]}{\mathcal{I}_\Sigma+J_{\Sigma,\beta}}=\mathrm{SR}(\Sigma,\beta).
  \end{align*}
  Thus $A^*(\mathcal{X}_{\widetilde{\Sigma},\widetilde{\beta}})\cong \mathrm{SR}(\Sigma,\beta)$ as graded $R_\beta$-algebras.
  Let $\pi\colon X_{\widetilde{\Sigma}}\to \mathrm{Spec}(k)$ be the structure morphism. 
  By \autoref{lemce}, $x_\rho$ corresponds to $c_1^{G_{\widetilde{\beta}}}(\pi^*\mathbb{A}^1_{\delta_{\widetilde{\beta}}(e_\rho^*)})$, which is the equivariant class $[V(X_\rho)]_{G_{\widetilde{\beta}}}$.
\end{proof}

If $\Sigma$ is smooth, then $\mathcal{I}_\Sigma=I_\Sigma$ and the canonical stack morphism is an isomorphism, so the isomorphism of \autoref{thmmain} follows.
Since $V(X_\rho)$ corresponds to $D_\rho$ via the Cox construction, $x_\rho$ corresponds to $[D_\rho]_{G_\beta}$ under this isomorphism.

\begin{remark}
  Let $(\Sigma,\beta\colon L\to N)$ be a non-strict stacky fan with $\Sigma$ smooth and $\mathrm{coker}(\beta)$ finite.
  Fix an isomorphism $N=\mathbb{Z}^d\oplus\mathbb{Z}/a_1\mathbb{Z}\oplus\cdots\oplus\mathbb{Z}/a_s\mathbb{Z}$, and let $\bar{\beta}\colon L\to \mathbb{Z}^d$ be induced by $\beta$. 
  Then $(\Sigma,\bar{\beta})$ is a strict stacky fan, and \cite[Lemma~A.3]{GS1} gives an exact sequence
  \[
  0\longrightarrow D(G_{\bar{\beta}})\longrightarrow D(G_\beta)\longrightarrow \bigoplus_{i=1}^s\mathbb{Z}/a_i\mathbb{Z} \longrightarrow 0.
  \]  
  Choose lifts $z_1,\dots,z_s\in D(G_\beta)$ of the standard generators of $\bigoplus_{i=1}^s\mathbb{Z}/a_i\mathbb{Z}$, and set $\chi_i= a_iz_i\in D(G_{\bar{\beta}})$.
  Then we have $R_\beta\cong R_{\bar{\beta}}[z_1,\dots,z_s]/(a_1z_1-\chi_1,\dots,a_sz_s-\chi_s)$ and
  \[
  \mathrm{SR}(\Sigma,\beta)\cong \frac{\mathrm{SR}(\Sigma,\bar{\beta})[z_1,\dots,z_s]}{(a_1z_1-\chi_1,\dots,a_sz_s-\chi_s)}.
  \]
  This is consistent with \cite[Proposition~7.20]{GS1}, which shows that every smooth non-strict toric stack is an essentially trivial gerbe over a smooth toric stack.
\end{remark}

\subsection{Functoriality}\label{subsecpull}

Let $(\Phi,\phi)\colon (\Sigma,\beta)\to (\Sigma',\beta')$ be a morphism of non-strict stacky fans with $\Sigma$ and $\Sigma'$ smooth.
In this subsection, we describe the homomorphism between the Stanley--Reisner rings induced by the pullback on equivariant Chow rings.

We first recall the pullback on equivariant Chow rings \cite{Kresch1999}.
Let $\rho\colon G\to H$ be a homomorphism of linear algebraic groups and $f\colon X\to Y$ a $\rho$-equivariant morphism of smooth varieties.
The induced morphism of quotient stacks factors as 
\[
[X/G]\stackrel{f_G}{\longrightarrow} [Y/G]\stackrel{p_Y}{\longrightarrow} [Y/H],
\]
where $f_G$ is representable and l.c.i., and $p_Y$ is smooth.
By \cite[Section~2.2 and~5.1]{Kresch1999}, the corresponding Gysin and flat pullbacks define 
\[
(\rho,f)^*=f_G^*\circ p_Y^*\colon A^*_H(Y)\longrightarrow A^*_G(X).
\]
For every $H$-equivariant line bundle $\mathcal{L}$ on $Y$, we have 
\[
(\rho,f)^*c_1^H(\mathcal{L})=c_1^G(f^*\rho^*\mathcal{L}).
\]

Let $g\colon G_\beta\to G_{\beta'}$ and $f_\Phi\colon X_\Sigma\to X_{\Sigma'}$ be the morphisms induced by $(\Phi,\phi)$.
Then the above pullback gives $(g,f_\Phi)^*\colon A^*_{G_{\beta'}}(X_{\Sigma'})\longrightarrow A^*_{G_\beta}(X_\Sigma)$.
By \autoref{thmmain}, this induces a homomorphism 
\[
\mathrm{SR}(\Sigma',\beta')\longrightarrow \mathrm{SR}(\Sigma,\beta).
\]

We next recall the pullback of ray divisors under a toric morphism. For each $\rho\in \Sigma(1)$, let $\sigma'_\rho$ be the unique cone of $\Sigma'$ whose relative interior contains $\Phi(u_\rho)$.
Since $\Sigma'$ is smooth, there are unique positive integers $a_{\rho,\rho'}$ for $\rho'\in\sigma'_\rho(1)$ such that 
\[
\Phi(u_\rho)=\sum_{\rho'\in \sigma'_\rho(1)} a_{\rho,\rho'}u_{\rho'}.
\]
Setting $a_{\rho,\rho'}=0$ for $\rho'\notin\sigma'_\rho(1)$, we have
\[
f_\Phi^*D_{\rho'}=\sum_{\rho\in\Sigma(1)}a_{\rho,\rho'}D_\rho.
\]

\begin{proposition}\label{propfun}
  Regard $\mathrm{SR}(\Sigma,\beta)$ as a graded $R_{\beta'}$-algebra via the homomorphism $R_{\beta'}\to R_\beta$ induced by $D(G_{\beta'})\to D(G_\beta)$, $\chi'\mapsto \chi'\circ g$.
  Then the induced homomorphism
  \[
  \mathrm{SR}(\Sigma',\beta')\longrightarrow \mathrm{SR}(\Sigma,\beta)
  \]
  is a graded $R_{\beta'}$-algebra homomorphism determined by 
  \[x_{\rho'}\longmapsto\sum_{\rho\in\Sigma(1)}a_{\rho,\rho'}x_\rho\]
  for $\rho'\in \Sigma'(1)$.
\end{proposition}
\begin{proof}
  We first consider the algebra structure.
  Let $p\colon X_\Sigma\to \mathrm{Spec}(k)$ and $p'\colon X_{\Sigma'}\to \mathrm{Spec}(k)$ be the structure morphisms.
  By the definition of the $R_{\beta'}$-algebra structure, $\chi'\in D(G_{\beta'})$ corresponds to $c_1^{G_{\beta'}}((p')^*\mathbb{A}^1_{\chi'})\in A^1_{G_{\beta'}}(X_{\Sigma'})$.
  The compatibility yields 
  \begin{align*}
    (g,f_\Phi)^*c_1^{G_{\beta'}}((p')^*\mathbb{A}^1_{\chi'})&=c_1^{G_\beta}(f^*_\Phi g^*((p')^*\mathbb{A}^1_{\chi'}))\\
                                                            &=c_1^{G_\beta}(p^*\mathbb{A}^1_{\chi'\circ g}),
  \end{align*}
  where the second equality follows from $p'\circ f_\Phi=p$ and $g^*\mathbb{A}^1_{\chi'}=\mathbb{A}^1_{\chi'\circ g}$.
  Thus the induced homomorphism sends $\chi'$ to $\chi'\circ g$, and it is a homomorphism of graded $R_{\beta'}$-algebras.

  Next we compute the images of the generators $x_{\rho'}$.
  For $\rho'\in \Sigma'(1)$, \autoref{thmmain} identifies $x_{\rho'}$ with $c_1^{G_{\beta'}}(\mathcal{O}(D_{\rho'}))$.
  Since $f_\Phi$ is $g$-equivariant, there is a $G_\beta$-equivariant isomorphism $f_\Phi^*g^*\mathcal{O}(D_{\rho'})\cong \mathcal{O}(f^*_\Phi D_{\rho'})$.
  Hence we have
  \begin{align*}
    (g,f_\Phi)^*c_1^{G_{\beta'}}(\mathcal{O}(D_{\rho'}))&=c_1^{G_\beta}(f_\Phi^*g^*\mathcal{O}(D_{\rho'}))\\
                                                        &=\sum_{\rho\in \Sigma(1)}a_{\rho,\rho'}c_1^{G_\beta}(\mathcal{O}(D_\rho)).
  \end{align*} 
  This proves the proposition.
\end{proof}

When $\beta=\mathrm{id}_L$ and $\beta'=\mathrm{id}_{L'}$, this agrees with \cite[Theorem~2.1]{PARK}.

\subsection{Comparison with previous results}\label{subseccom}

\subsection*{Strict toric stacks}

Let $(\Sigma,\beta\colon L\to N)$ be a strict stacky fan with $\Sigma$ smooth.
Since $D(G_\beta)\cong L^*/\beta^*(N^*)$, \autoref{thmmain} gives 
\[
A^*(\ts)\cong \frac{\mathbb{Z}[x_\rho\,\mid\,\rho\in\Sigma(1)]}{I_\Sigma+\left(\sum_{\rho\in\Sigma(1)}\langle \beta^*(m),u_\rho\rangle x_\rho\,\middle|\,m\in N^*\right)}.
\]
This agrees with \cite[Theorem~7.2.5]{Singh2026}.
In particular, for $\beta=\mathrm{id}_L$, we recover the classical Stanley--Reisner presentation of $A^*(X_\Sigma)$.
The projective case was proved by Jurkiewicz \cite{Jurkiewicz}, and the complete case by Danilov \cite{VIDanilov_1978}.
For an arbitrary smooth fan, this presentation is stated in \cite{Gross}.

This presentation also applies to fantastacks.
Let $\Sigma$ be a fan on a lattice $N$, and let $\beta\colon \mathbb{Z}^n\to N$ be a homomorphism with finite cokernel.
We assume that $b_i=\beta(e_i)$ lies in the support of $\Sigma$ for every $i$ and that each ray of $\Sigma$ contains some $b_i$.
For each $\sigma\in\Sigma$, set $\widehat{\sigma}=\mathrm{Cone}(e_i\mid b_i\in\sigma)$.
The cones $\widehat{\sigma}$ generate a fan $\widehat{\Sigma}$ on $\mathbb{Z}^n$, and the corresponding \textit{fantastack} is $\mathcal{F}_{\Sigma,\beta}=\mathcal{X}_{\widehat{\Sigma},\beta}$.
The formula gives 
\[
A^*(\mathcal{F}_{\Sigma,\beta})\cong\frac{\mathbb{Z}[x_1,\dots,x_n]}{I_{\widehat{\Sigma}}+\left(\sum_{i=1}^n\langle m,b_i\rangle x_i\,\middle|\,m\in N^*\right)}.
\]
If each $e_i$ corresponds to a ray of $\Sigma$, then this formula coincides with the formula of \cite[Proposition~2.2]{integration}.
If, in addition, each $b_i$ is the primitive generator, it recovers \cite[Theorem~3.6]{ER} for canonical stacks.

\subsection*{Toric Deligne--Mumford stacks}

Let $\boldsymbol{\Delta}=(N,\Delta,\beta)$ be a stacky fan in the sense of \cite{Borisov_2004}, with rays $\rho_1,\dots,\rho_n$ and $\beta\colon \mathbb{Z}^n\to N$.
From the construction, we have $\mathcal{X}(\boldsymbol{\Delta})=\mathcal{X}_{\widetilde{\Delta},\beta}$, where $\widetilde{\Delta}$ is the fan on $\mathbb{Z}^n$ generated by $\mathrm{Cone}(e_i\mid\rho_i\subset\sigma)$ for $\sigma\in\Delta$.
Since the primitive ray generators of $\widetilde{\Delta}$ are the standard basis vectors, \autoref{thmmain} yields 
\[
A^*(\mathcal{X}(\boldsymbol{\Delta}))\cong \frac{R_\beta}{\left(\delta_\beta(e^*_{i_1})\cdots\delta_\beta(e^*_{i_s})\,\middle|\,\rho_{i_1},\dots,\rho_{i_s}\ \text{do\ not\ span\ a\ cone\ of}\ \Delta\right)}.
\]

Suppose that $N_\mathrm{tor}\subset \mathrm{im}(\beta)$. 
From \cite[Lemma~2.6]{jiang2007integralorbifoldchowring}, we obtain $R_\beta\cong \mathbb{Z}[x_1,\dots,x_n]/C(\boldsymbol{\Delta})$, where $C(\boldsymbol{\Delta})$ is the ideal generated by linear relations $\sum_{i=1}^n\langle m,\beta(e_i)\rangle x_i$ for $m\in N^*$.
Let $\tilde{x}_i$ be the element corresponding to $\delta_\beta(e_i^*)$ under this isomorphism.
Therefore, 
\[
A^*(\mathcal{X}(\boldsymbol{\Delta}))\cong\frac{\mathbb{Z}[x_1,\dots,x_n]}{\widetilde{I}_{\boldsymbol{\Delta}}+C(\boldsymbol{\Delta})},
\]
where $\widetilde{I}_{\boldsymbol{\Delta}}$ is the ideal generated by $\tilde{x}_{i_1}\cdots\tilde{x}_{i_s}$ for all $\rho_{i_1},\dots,\rho_{i_s}$ that do not span a cone of $\Delta$.
This is the presentation in \cite[Theorem~1.1]{jiang2007integralorbifoldchowring} and \cite[Proposition~4.9]{CE}.

\subsection{Example}

Consider the following morphism of non-strict stacky fans
\[
\begin{tikzpicture}[>=stealth]
  \node at (-2.8,1.35) {$\Sigma=$};
  \begin{scope}[shift={(-2.1,.8)},scale=.5]
    \fill[gray!16] (0,0) -- (3.5,0) -- (3.5,3.5) -- cycle;
    \fill[gray!26] (0,0) -- (3.5,3.5) -- (0,3.5) -- cycle;
    \draw[help lines,step=1] (-.5,-.5) grid (3.5,3.5);
    \draw[->] (0,0) -- node[pos=.72,below] {$\scriptstyle\rho_1$} (3.5,0);
    \draw[->] (0,0) -- node[pos=.72,left] {$\scriptstyle\rho_3$} (0,3.5);
    \draw[->] (0,0) -- node[pos=.78,right] {$\scriptstyle\rho_2$} (3.5,3.5);
  \end{scope}
  \node at (1.35,1.35) {$\Sigma'=$};
  \draw[->] (2.55,1.35) -- node[pos=.72,above] {$\scriptstyle\rho'_1$} (3.15,1.35);
  \draw[->] (2.55,1.35) -- node[pos=.72,above] {$\scriptstyle\rho'_2$} (1.95,1.35);
  \fill (2.55,1.35) circle (1.2pt);
  \node (L) at (-1.55,0) {$L=\mathbb{Z}^2$};
  \node (Lp) at (2.55,0) {$L'=\mathbb{Z}$};
  \node (N) at (-1.55,-1.3) {$N=\mathbb{Z}\oplus\mathbb{Z}/2\mathbb{Z}$};
  \node (Np) at (2.55,-1.3) {$N'=\mathbb{Z}$\rlap{.}};
  \draw[->] (L) -- node[above] {$\scriptstyle\Phi=\begin{psmallmatrix}2&-2\end{psmallmatrix}$} (Lp);
  \draw[->] (L) -- node[left] {$\scriptstyle\beta=\begin{psmallmatrix}4&-4\\1&0\end{psmallmatrix}$} (N);
  \draw[->] (Lp) -- node[right] {$\scriptstyle\beta'=2$} (Np);
  \draw[->] (N) -- node[below]  {$\scriptstyle\phi=\begin{psmallmatrix}1&0\end{psmallmatrix}$} (Np);
\end{tikzpicture}
\]
The toric varieties associated with $\Sigma$ and $\Sigma'$ are $X_\Sigma=\mathrm{Bl}_0\mathbb{A}^2$ and $X_{\Sigma'}=\mathbb{P}^1$, respectively.
From this morphism, diagram~\eqref{equdiag} becomes:
\[
\begin{tikzpicture}[auto,->,hookarrow/.style={arrows={Hooks[right]->}},baseline={([yshift=-0.5ex]current bounding box.center)}]
  \node (1) at (-1,1.5) {$\mathbb{Z}$};
  \node (2) at (2.1,1.5) {$\mathbb{Z}$};
  \node (3) at (5.15,1.5) {$\mathbb{Z}/2\mathbb{Z}$};
  \node (4) at (7.25,1.5) {$0$};
  \node (5) at (-1,0) {$\mathbb{Z}^2$};
  \node (6) at (2.1,0) {$\mathbb{Z}^3$};
  \node (7) at (5.15,0) {$\mathbb{Z}\oplus\mathbb{Z}/4\mathbb{Z}$};
  \node (8) at (7.25,0) {$0$};
  \node (9) at (2.1,-1) {$L^*$};
  \draw (1) -- node {$\scriptstyle \beta'^*=(2)$} (2);
  \draw (2) -- node {$\scriptstyle \delta_{\beta'}=(1)$} (3);
  \draw (3) -- (4);
  \draw (5) -- node[swap] {$\scriptstyle \begin{psmallmatrix} 4&1\\ -4&0\\ 0&2 \end{psmallmatrix}$} (6);
  \draw (6) -- node[above] {$\scriptstyle \begin{psmallmatrix} 2&2&-1\\ 0&1&0 \end{psmallmatrix}$} (7);
  \draw (7) -- (8);
  \draw (1) -- node[left] {$\scriptstyle \begin{psmallmatrix}1\\0\end{psmallmatrix}$} (5);
  \draw (2) -- node[left] {$\scriptstyle F^*=\begin{psmallmatrix}2\\-2\\0\end{psmallmatrix}$} (6);
  \draw (3) -- node[right] {$\scriptstyle D(g)=\begin{psmallmatrix}0\\2\end{psmallmatrix}$} (7);
  \draw (9) -- node[swap] {$\scriptstyle \delta_\beta=\begin{psmallmatrix}2&2\\0&1\end{psmallmatrix}$} (7);
  \draw[hookarrow] (9) -- (6);
\end{tikzpicture}
\]
It follows that $G_\beta\cong \mathbb{G}_m\times \mu_4$ and $G_{\beta'}\cong \mu_2$, and hence the corresponding toric stacks are 
\[
\ts=[\mathrm{Bl}_0\mathbb{A}^2/(\mathbb{G}_m\times \mu_4)],\quad \mathcal{X}_{\Sigma',\beta'}=[\mathbb{P}^1/\mu_2],
\]
where the actions are induced by $(\lambda,\zeta)\cdot(z_1,z_2)=(\lambda^2 z_1,\lambda^2\zeta z_2)$ and $\epsilon\cdot[w_1:w_2]=[\epsilon w_1:w_2]$.

Since $R_\beta\cong \mathbb{Z}[s,t]/(4t)$ and $R_{\beta'}\cong \mathbb{Z}[u]/(2u)$, \autoref{thmmain} yields
\[
A^*(\ts)\cong\mathrm{SR}(\Sigma,\beta)=\frac{\mathbb{Z}[s,t,x_1,x_2,x_3]}{(4t,x_1x_3,x_1+x_2-2s,x_2+x_3-2s-t)}
\]
and
\[
A^*(\mathcal{X}_{\Sigma',\beta'})\cong\mathrm{SR}(\Sigma',\beta')=\frac{\mathbb{Z}[u,y_1,y_2]}{(2u,y_1y_2,y_1-y_2-u)},
\]
where $x_i=x_{\rho_i}$ and $y_j=x_{\rho'_j}$.
The homomorphism $R_{\beta'}\to R_\beta$ induced by $D(g)$ is given by $u\mapsto 2t$, and the equalities 
$\Phi(u_{\rho_1})=2u_{\rho'_1}$, $\Phi(u_{\rho_2})=0$, and $\Phi(u_{\rho_3})=2u_{\rho'_2}$
yield
$f_\Phi^*D_{\rho'_1}=2D_{\rho_1}$ and $f_\Phi^*D_{\rho'_2}=2D_{\rho_3}$.
Therefore, by \autoref{propfun}, the homomorphism $\mathrm{SR}(\Sigma',\beta')\to \mathrm{SR}(\Sigma,\beta)$ is given by 
\[
u\longmapsto2t,\qquad y_1\longmapsto2x_1,\qquad y_2\longmapsto2x_3.
\]
Since $\ts$ is not Deligne--Mumford and has non-trivial generic stabilizer, it is not covered by the previous results in \autoref{subseccom}.

\bibliographystyle{alpha}
\bibliography{AGmath}

\end{document}